\documentclass{article}
\usepackage[T1]{fontenc}
\usepackage{times,mathptm,  hyperref}
\usepackage{amssymb,epsfig,verbatim,xypic, amsthm}

\newtheorem{thm}{Theorem}[section]
\newtheorem{conj}{Conjecture}[section]

\newtheorem{lem}[thm]{Lemma}

\newtheorem{que}{Question}
\newtheorem{eg}[thm]{Example}
\newtheorem{rem}[thm]{Remark}

\newenvironment{thm-A}
{\noindent{\bf Theorem A.--}\it}{\\}

\newenvironment{thm-M}
{\noindent{\bf Main Theorem.}\it}{\\}

\newenvironment{thm-AA}
{\noindent{\bf Theorem A'.}\it}{\\}

\newenvironment{thm-B}
{\noindent{\bf Theorem B.--}\it}{\\}

\newenvironment{thm-C}
{\noindent{\bf Theorem C.--}\it}{\\}

\newenvironment{thm-BP}
{\noindent{\bf Bell-Poonen Theorem.--}\it}{\\}

\newenvironment{thm-BB}
{\noindent{\bf Theorem B'.}\it}

\def\R{\mathbf{R}}

\def\Z{\mathbf{Z}}

\addtocounter{section}{0}             

\newcommand{\Addresses}{{
  \bigskip
  \footnotesize

  O.~Paris-Romaskevich \textsc{Aix Marseille Univ, CNRS, Centrale Marseille, I2M, Marseille, France}\par\nopagebreak
  \textit{E-mail address} O.~Paris-Romaskevich: \texttt{olga.romaskevich@math.cnrs.fr}

}}

\begin{document}

%
%
\title{Tiling billiards and Dynnikov's helicoid}
\date{}
\author{Olga Paris-Romaskevich}

\maketitle
{\emph{To Anatoly Stepin, who helped me do my first steps as a researcher.}}

\begin{abstract} 
Here are two problems.  First,  understand the dynamics of a tiling billiard in a cyclic quadrilateral periodic tiling.  Second,  describe the topology of connected components of plane sections of a centrally symmetric subsurface $S \subset \mathbb{T}^3$ of genus $3$.  In this note we show that these two problems are related via a helicoidal construction proposed recently by Ivan Dynnikov.  The second problem is a particular case of a classical question formulated by Sergei Novikov. The exploration of the relationship between a large class of tiling billiards (periodic locally foldable tiling billiards) and Novikov's problem in higher genus seems promising,  as we show in the end of this note.

\emph{Bibliography : $25$ items; $5$ figures; MSC: Primary 37E35,  Secondary 37J60; keywords : Novikov's problem,  tiling billiards,  billiards,  translation surfaces}
\end{abstract}

\section{Introduction and intentions}
Tiling billiards are billiards in tilings.  They were first introduced only several years ago,  in the works of Davis and her coauthors,  see \cite{DavisNegative2018, Davis2018,  DavisHooper2016}.  The definition of the billiard flow is as follows.  Each time a ray of light crosses an edge between two tiles,  it refracts through this edge.  A new direction of the beam is obtained from the old one by reflection with respect to the crossed edge,  following thus the Snell's law of refraction with coefficient $-1$,  see Figure \ref{fig:refraction_law}.  

The goal is to understand the dynamics of such tiling billiards.  What are typical  trajectories? And atypical ones? The answers to these questions,  and the dynamics in general,  depend strongly on the form of the underlying tiling.  

\begin{figure}
\centering
\includegraphics[scale=0.16]{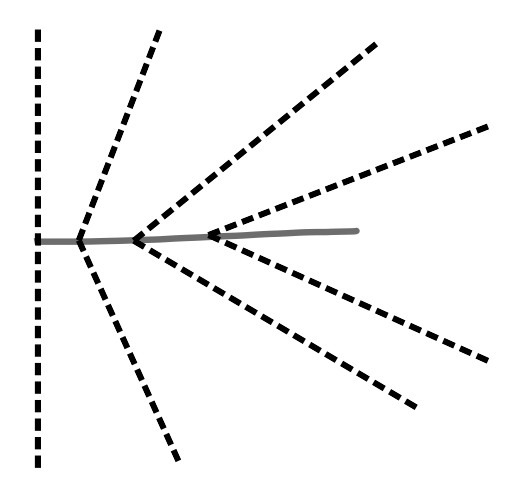}
\caption{Four beams of light crossing a horizontal line under tiling billiard law. }\label{fig:refraction_law}
\end{figure}

To this date,  our community has reached relative success in the understanding of the dynamics of three non-trivial tiling billiards.  These are,   \emph{trihexagonal tiling}  \cite{DavisHooper2016}; \emph{periodic triangle tilings}
\cite{Davis2018},   \cite{HubertPaRo2019},  \cite{PaRo2019}; 
and \emph{periodic cyclic quadrilateral tilings} \cite{DHMPRS2020}.  In particular,  the dynamics of a tiling billiard in a parallelogram tiling seems at the moment completely obscure.

One of the reasons to be interested in tiling billiards is their connection to classical  objects in mathematics.  For example,  the dynamics of triangle tiling billiards is equivalent to that  of Arnoux-Rauzy family of interval exchange transformations and their rel-deformations.  The exceptional set of trajectories of these tilings is parametrized by a famous fractal object,  the \emph{Rauzy gasket},  see \cite{ArnouxStarosta2013} for its definition. 

In this work we point out a new connection of triangle and cyclic quadrilateral tiling billiards to another classical subject which is a so-called Novikov's problem.  This problem studies the connected components of plane sections of triply periodic surfaces.  Of particular interest are chaotic components -- curves such that there closure in the fundamental domain of the surface fills in the subsurface of genus at least $3$.  Novikov's problem has deep connections to conductivity physics and is often presented as the problem on semiclassical motion of an electron in a homogeneous magnetic field.  We send our reader to \cite{NoDyMa} for the overview of the state of art on Novikov's problem from the point of view of experimental physics.  In this note we formulate and study this problem in purely topological terms.  Mathematically,  first observations were done by Zorich in \cite{Zorich},  and a breakthrough via Morse theory has been done by Dynnikov,  we refer especially to \cite{Dynnikov1999TheGO}.  The generalizations of the problem related to quasi-periodic functions having more than $3$ quasiperiods appear in \cite{DyNoquasi}.

The connection between tiling billiard systems and Novikov's problem has been hinted to us by Dynnikov.  His constuction gives a title to this work,  and we make it explicit.

Once the connection between two subjects,  tiling billiards and Novikov's problem,  is established,  we show how the ideas from topology and Morse theory that were developped for treating the Novikov's problem,  apply to tiling billiard dynamics.  We formulate,  at the very end of this work,  Conjecture \ref{conj:probability} on the behavior of tiling billiards in a large class of tilings that includes triangle and cyclic quadrilateral tilings - \emph{locally foldable tilings}.  This conjecture is based on two strongly non-trivial results.  First,  Dynnikov's result on the generic behaviour of plane sections of $3$-periodic surfaces that he obtained at the end of the last century in \cite{Dynnikov1999TheGO}.  And second,  a recent result by Kenyon,  Lam,  Ramassamy and Russkikh \cite{Kenyon2018DimersAC} describing the space of parameters of locally foldable tilings,  in the setting of the dimer model.

\smallskip

The general intention of this article is to express hope for a new approach of the Novikov's problem through tiling billiard dynamics and,  in particular,  renormalization for such dynamics.  We refer our reader to  \cite{PaRo2019} and \cite{DHMPRS2020} for introduction to renormalization in tiling billiards.  Recently we have shown,  in collaboration with Dynnikov,  Hubert,  Mercat and Skripchenko,  that the measure of chaotic regimes in the Novikov's problem with central symmetry,  in genus $3$,  is equal to $0$.  This question has been open for $40$ years and we answer it using the renormalization of cyclic quadrilateral tiling billards.  We hope that the connection with tiling billiards will permit new discoveries for Novikov's problem in higher genus as well.

\bigskip

This note is structured in a following way.  In Section \ref{sec:folding},  we remind the folding procedure for tiling billiards and define locally foldable tilings,  and,  in particular,  triangle and cyclic quadrilateral periodic tilings.  The billiards in these two last classes of tilings are the main dynamical systems discussed in this note.  The Section \ref{sec:folding} introduces the folding map and the so-called parallel foliations are obtained as preimages under folding of standard foliations by parallel lines.  In Section \ref{sec:Dynnikov} triply periodic surfaces called Dynnikov's helicoids are constructed,  correponding to tiling billiards.  In Section \ref{sec:Novikov} we remind the statement of Novikov's problem (paragraph \ref{subs:novikovs}) and make explicit the connection between tiling billiards and this problem,  using Dynnikov's helicoids.  Then,  we interpret classical results on Novikov's problem in terms of such tiling billiards (paragraph \ref{subs:interpretation}) and advance in the proof of the so-called Tree Conjecture for cyclic quadrilateral tilings (paragraph \ref{subs:Tree Conjecture 4}).  Sections \ref{sec:Dynnikov} and \ref{sec:Novikov} concern only triangle and cyclic quadrilateral tilings.  
In Section \ref{sec:generalisations_last},  we discuss open questions and perspectives for general locally foldable tilings.

\smallskip

We excuse ourselves for some familiarity with regularity.  The goal of this article is to share the ideas on a conceptual level.  The question of regularity of the surfaces in Novikov's problem is although extremely important,  and should be addressed if one wants to obtain precise statements.  It will be done in the upcoming work \cite{DHMPRS2020}, at least for cyclic quadrilateral tilings.

\section{Folding and its consequences}\label{sec:folding}
Any triangle (or quadrilateral) $P$ tiles a plane periodically as follows.  A fundamental domain $\mathcal{D}_m$ of a tiling is obtained by gluing $P$ with its centrally symmetric copy, with respect to a middlepoint $m$ of its side.  We call a corresponding $P$\textbf{-tiling} $\mathcal{P}$ a \textbf{(periodic) triangle (quadrilateral) tiling}.  Such tiling is $2$\textbf{-colorable} in a way that neighbouring tiles have different colors,  as a chess-board.  We denote $\mathcal{P}$ a tiling and the set of its tiles.

A tiling of a plane by polygons is \textbf{locally foldable} if it is $2$-colorable,  and the sum of angles in any vertex $v$ of the same-colored tiles containing $v$ is equal to $\pi$.  Triangle tilings are locally foldable,  and quadrilateral tilings are locally foldable only if the quadrilateral is \textbf{cyclic} (inscribed in a circle).  Tiling billard trajectories in locally foldable tilings share fundamental properties that we state in the points $2. $ and $3. $ of Theorem \ref{thm:first_properties} below.  Throughout this work,  we concentrate on the case of $P$-tilings by triangles and cyclic quadrilaterals.  In Section \ref{sec:generalisations_last} we discuss the general case.

\subsection{Folding and reduction of dynamics to dimension one}\label{subs:first}

\emph{I like to fold my magic carpet, after use, in such a way as to superimpose one part of the pattern upon another. }
Vladimir Nabokov, \emph{ Speak,  Memory}

\smallskip

Opening up a polygonal billiard table in order to understand the trajectories is now a habit for any mathematician : while a table is unfolded in order to produce a,  potentially,  self-overlapping and layered tiling of the plane,  a trajectory is unfolded into a straight line.  This idea (called Katok-Zemlyakov construction) provides,  in the case of rational tables,  a connection of billiard dynamics with translation flows.\footnote{For tiling billiards,  the connection with translation flows is not straightforward,  even though it may be given,  at least for locally foldable tilings. }

For tiling billiards,  the billiard table is \emph{already} a tiled plane.  We fold two neighbouring tiles along the crease like the wings of an asymmetric butterfly.  The segments of a tiling billiard trajectory in these two tiles fold into segments on the same line.  It is easy to show that for locally foldable tilings,  such a \textbf{folding map} is defined globally,  and not only along a path in a tiling.  The entire plane may be folded and any tiling billiard trajectory folds into a line inside this folding,  see \cite{PaRo2019} for more details and precise statements.  In the following,  we use this folding map,  unique up to isometry.

For triangle and cyclic quadrilateral tilings,  the image of a folded plane is particularly simple to understand.  The plane folds inside a disk,  and all of the vertices of a tiling fold onto its boundary - a circle.  The preimage of this circle under the folding is the union of all circumcircles of tiles.  This beautiful observation first appeared in the work \cite{Davis2018} by Baird-Smith,  Davis,  Fromm and Iyer.  Their work contains many illustrations,  as well as a pattern to cut out and experience the folding manually.

\smallskip

A powerful,  and elementary,  consequence of the existence of global folding is

\begin{thm}[\cite{Davis2018}, \cite{HubertPaRo2019},  \cite{PaRo2019}]\label{thm:first_properties}
The following holds for the trajectories of tiling billiards in periodic triangle and cyclic quadrilateral tilings.
\begin{enumerate}
\item[1.] The oriented distance $\tau(\gamma,  P)$  between an (oriented) segment of a trajectory $\gamma$ in a tile $P$ and the circumcenter of this tile,  is constant along $\gamma$,  that is,  $\tau(\gamma,  P)=\tau(\gamma)$;
\item[2.] every trajectory intersects any tile in at most one segment;
\item[3.] any bounded trajectory is periodic and stable under small perturbations (of a tile or initial condition): perturbed trajectory passes by the same tiles.
\end{enumerate}
\end{thm}

\smallskip

The authors of \cite{Davis2018} use the folding in order to reduce the dynamics of the triangle tiling billiard to that of a family of $3$-interval exchange transformations (with flips!) on a circle.  For cyclic quadrilaterals,  similarly,  one gets $4$-IETs,  see paragraph \ref{subs:sym_and_IET} here.   
In Section $3$ of our work \cite{PaRo2019},  Theorem \ref{thm:first_properties} is generalized for all locally foldable tilings.
 
\subsection{Parallel foliations of trajectories}\label{subs:parallel}
Take any periodic trajectory $\gamma$ of a billiard in a tiling,  not even necessarily locally foldable.  A close enough to $\gamma$ trajectory $\gamma'$,  launched \emph{in the same direction} as $\gamma$,  is necessarily periodic and disjoint from $\gamma$.  The cylinder between $\gamma$ and $\gamma'$ is foliated by parallel tiling billiard trajectories.

For triangle and cyclic quadrilateral tiling billiards this idea can be pushed much further.   First,  as follows from Theorem \ref{thm:first_properties},  «in the same direction» can be omitted for triangle and cyclic quadrilateral tilings.  And second,  parallel trajectories foliate the full plane,  as we noticed and explored in \cite{PaRo2019}.  For any trajectory $\gamma$,  there exists a so-called  \textbf{parallel foliation} of the entire tiled plane (singular only in vertices of the tiling) such that its non-singular leaves are tiling billiard trajectories (one of which is $\gamma$) and that in restriction to any tile,  it is a foliation by parallel segments.  We invite our reader to discover a video work \emph{Refraction tilings} by O.  David on parallel foliations.\footnote{The video is accessible on the Youtube channel \emph{Dragonazible,}  or via \href{https://youtu.be/t1r1cO1V35I}
{\underline{https://youtu.be/t1r1cO1V35I}}.}

\begin{rem}
It is interesting to compare the parallel foliation construction for tiling billiards with the straight skeleton method for polygons \cite{Aichholzer} discovered in 1995.  This method was used by Demaines,  father and son,  and Lubiw,  see
\cite{Demaine},  in order to solve a following fold-and-cut problem.  Fix a polygonal motive on a piece of paper.  Can the paper be folded in such a way that a polygonal motive may be cut  out by exactly one scissor cut ? (and the answer is yes!)
\end{rem}

Parallel foliations are a handy tool to study the dynamics of locally foldable tiling billiards.  Indeed,  to one trajectory is associated an entire family of trajectories in the same parallel foliation.  The singular leaves of such a foliation uniquely define the symbolic dynamics of all other leaves.  We have introduced parallel foliations in \cite{PaRo2019} in order to prove the following Tree Conjecture (a theorem since),  formulated in \cite{Davis2018}.

\begin{thm}[\cite{PaRo2019}]\label{thm:tree_conjecture}
For any periodic trajectory $\gamma$ of a triangle tiling billiard,  the domain $\Omega_{\gamma}$ bounded by it doesn't contain any full tile of the tiling.  In other words,  all the vertices and edges of the tiling contained in $\Omega_{\gamma}$ form a graph which is a tree.
\end{thm}

In paragraph \ref{subs:Tree Conjecture 4},  we advance in the proof of the Tree Conjecture for cyclic quadrilateral tiling billiards.  For this,  we use a new idea that connects such tiling billiards with the topology of sections of periodic surfaces.  This idea is a heart of this article and deserves a section for itself. 

\section{Stairway to topology}\label{sec:Dynnikov}
Here we construct a helicoidal surface such that the connected components of its horizontal sections coincide with tiling billiard trajectories of the same parameter $\tau(\gamma)=\tau$,  see Theorem \ref{thm:first_properties}.  This surface will be triply periodic and,  via projection,  a compact subsurface of the $3$-torus.  The study of trajectories will be thus reduced to the study of plane sections of surfaces in $\mathbb{T}^3$ which is a classical problem discussed in Section \ref{sec:Novikov}.  We speculate in \cite{PaRo2019} on the existence of such a link between this problem and tiling billiards.  Recently,  Dynnikov made this connection precise,  and we make his brilliant idea explicit.

\subsection{Notations}\label{subsubs:notations}
For a triangle (or cyclic quadrilateral) tiling $\mathcal{P}$,  we fix the following notations that are respected throughout the article,  in particular on Figure \ref{fig:quasi_periodic}.

\smallskip

\textbf{Triangle tiling.} The sides of the tiles are $a,  b$ and $c$,  a clockwise tour of a tile reads $abc$.  The vertices opposite to the sides $a,b,c$ are $A,B,C$ and their angles are $\alpha,  \beta$ and $\gamma$.  Define vectors $\textbf{a}:=BC,  \textbf{b}:=CA,  \textbf{c}:=AB$.  Then $\textbf{a}+\textbf{b}+\textbf{c}=0$ and $\alpha+\beta+\gamma=\pi$.

\smallskip

\textbf{Cyclic quadrilateral tiling.} The sides of the tiles are $a,  b,  c$ and $d$,  a clockwise tour of a tile reads $abcd$.  The vertices are $A,B,C,D$ and $A=d \cap a,  B= a \cap b$ etc.  The angles in these vertices are $\alpha,  \beta,  \gamma$ and $\delta$.  Define vectors $\textbf{a}:=AB,  \textbf{b}:=BC,  \textbf{c}:=CD,  \textbf{d}:=DA$.  Cyclicity is equivalent to the relations $\alpha+\gamma=\beta+\delta=\pi$.  

\subsection{Quasi-periodicity of trajectory angles}
Fix a tile $P_0 \in \mathcal{P}$.  There exists a unique folding,  in sense of paragraph \ref{subs:first} such that the tile $P_0$ is fixed.  Then,  to any tiling billiard trajectory $\gamma$,  we associate two parameters.  

First,  we define the \textbf{energy} $\tau(\gamma)$ via Theorem \ref{thm:first_properties}.  It is positive if the trajectory turns counterclockwise around the circumcenter and negative otherwise.  When $\tau(\gamma)=0$,  the direction of the trajectory points in (or out of) the circumcenter.  The set of energy values is bounded and symmetric.  Up to a rescaling of the tiling,  we suppose that $\tau \in [-\pi,  \pi]$.

The second parameter is the \textbf{angle} $\theta(\gamma,  P)$ that makes an (oriented) segment of a trajectory in a tile $P$ with some fixed direction.  This parameter depends on a tile that the trajectory $\gamma$ crosses.  We denote by $\theta(\gamma):=\theta(\gamma,  P_0)$.  If some trajectory $\gamma$ doesn't cross a tile $P$,  the value $\theta(\gamma,  P)$ can still be defined and is done via the following

\begin{lem}\label{lem:lem_function}
For any tile $P_0$ in triangle (or cyclic quadrilateral) tiling,  there exists a unique function $f_P(\theta):=f(P,  \theta): \mathcal{P}\times [0,  2\pi) \rightarrow [0,  2\pi)$ such that $f(P_0,  \theta) \equiv \theta$ and that the following holds.
\begin{itemize}
\item[1.] For two tiles $P_1$ and $P_2$ of same (different) color,
$$f_{P_1}(\theta) \mp f_{P_2}(\theta)=\varphi(\textbf{v}),  \;\; \textbf{v}:=P_1-P_2.$$

Here $\varphi$ is a function of the vector $\textbf{v}$ that connects the barycenters of the tiles. \footnote{In other words,  this vector depends only on the relative positions of tiles.  This vector belongs to the period lattice for tiles of the same color.}  
\item[2.] For any trajectory $\gamma$ and any two tiles $P_1$ and $P_2$ that it crosses,  
$$\theta(\gamma,  P_1) \mp \theta(\gamma,  P_2)=\varphi(\textbf{v}).$$
Hence $\theta(\gamma,P)$ is globally and correctly defined for all $\gamma$ and all $P\in \mathcal{P}$,  even if $\gamma \cap P = \emptyset$.

\item[3.] The function $f$ is defined via Figure \ref{fig:quasi_periodic} in the case when the fixed direction is that of $AB$.  In general,  it is sufficient to define its values on all the tiles neighbouring to $P_0$ and then to continue by quasiperiodicity on all $\mathcal{P}$.
\end{itemize}
\end{lem}

\begin{proof}
This follows obviously from the existence of folding.  The images of two tiles of the same color $P_1,  P_2$ map under folding differ by a circle rotation by $\varphi(P_1-P_2)$. 
\end{proof}

The function $f_P(\theta)$ is described also in Table 1 of \cite{Davis2018} for triangle tilings.  Although,  there it is only defined on «half» of the tiles.  On Figure \ref{fig:quasi_periodic} we picture the values of the function $f$ on the tiles that one can access in one or two steps from $P_0$.  A Lemma analogous to Lemma \ref{lem:lem_function} can be proven for any periodic locally foldable tiling.  

\begin{figure}
\centering
\includegraphics[scale=0.45]{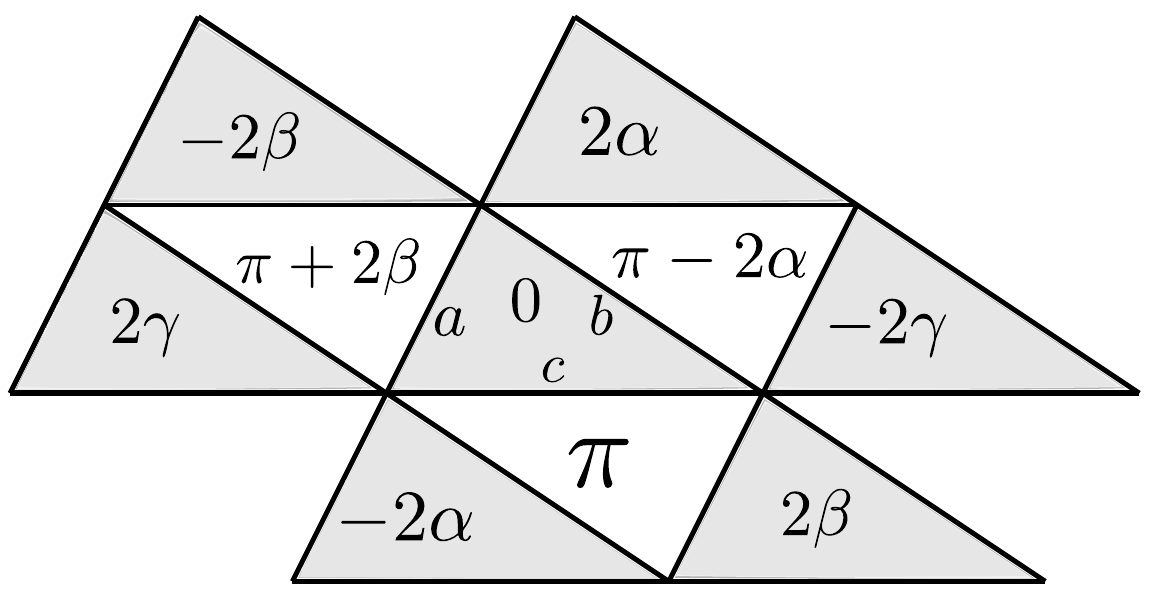}
\includegraphics[scale=0.6]{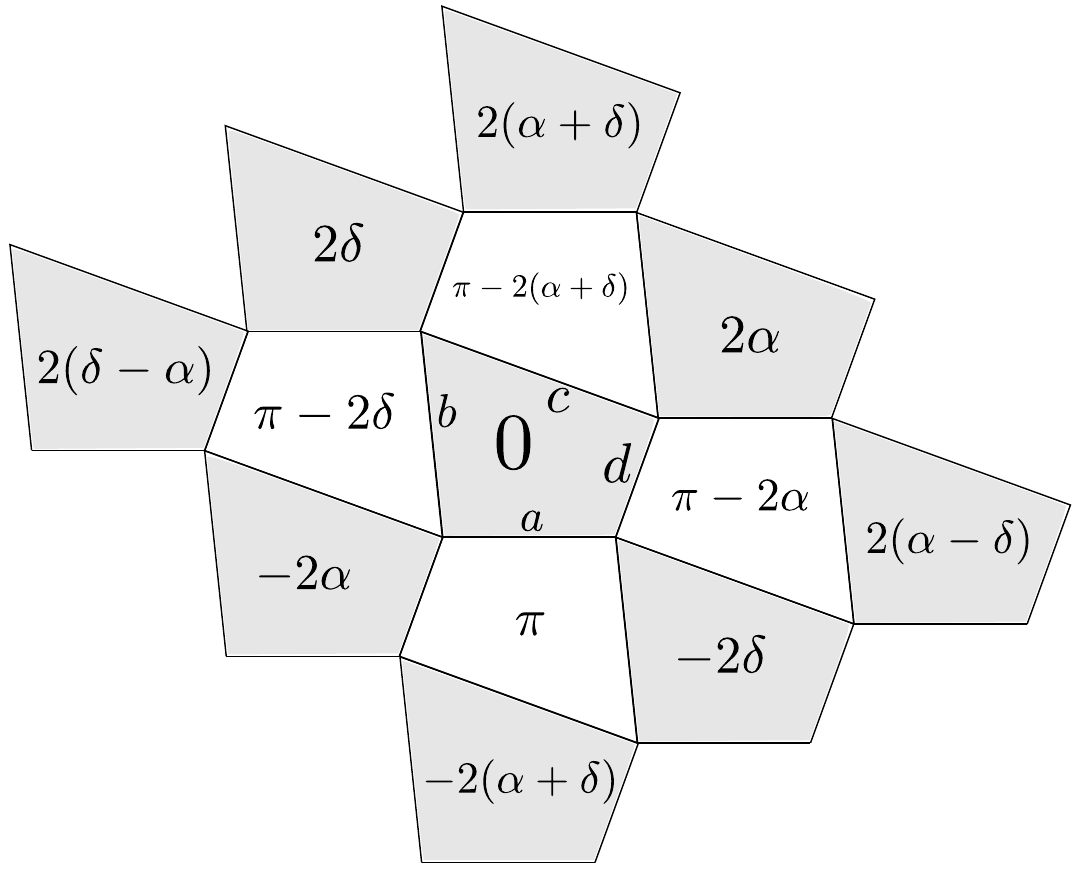}
\caption{The values of function $f_P(\theta) \pm f_{P_0}(\theta)$ in the vicinity of $P_0$,  for triangle and cyclic quadrilateral tilings.  The values are calculated using the notations of paragraph \ref{subsubs:notations}.  In both cases,  a fixed direction is chosen as that of the segment $AB$ in $P_0$.  }\label{fig:quasi_periodic}
\end{figure}

\begin{rem}
If a fixed direction is defined by some angle $\theta_0$ then,  on Figure \ref{fig:quasi_periodic},  one should add $\theta_0$ to all values of $\varphi(P-P_0)$ with $P$- grey and subtract it when $P$-white.
\end{rem}

\subsection{Dynnikov's helicoid: construction}\label{subsubs:Dynnikov}
In this paragraph we construct a one-parametric family $\{
\widehat{S_{\mathcal{P}}^{\tau}}
\}_{\tau \in [-\pi,\pi]}$ of piecewise smooth surfaces in $\R^3$,  associated to any triangle (cyclic quadrilateral) tiling $\mathcal{P}$.  All the trajectories with the same energy $\tau$ appear as horizontal sections of $\widehat{S_{\mathcal{P}}^{\tau}}$.

\smallskip

Fix some parameters $\tau \in [-\pi,\pi]$ and $\theta \in [0,2\pi)$.  Let $\Gamma_{\tau,\theta}$ be the set of all tiling billiard trajectories $\gamma$ such that $\tau(\gamma)=\tau$ and $\theta(\gamma,  P_0)=\theta$.  We identify the set $\Gamma_{\tau,\theta}$ and the geometric union of the trajectories $\cup_{\gamma \in \Gamma_{\tau,\theta}} \gamma$.  Potentially (and generically,  as shown in \cite{HubertPaRo2019},  \cite{PaRo2019} and \cite{DHMPRS2020}),  the set $\Gamma_{\tau,\theta}$ consists of more than one trajectory.  This set is obtained as the union of curves which fold into the same chord defined by parameters $\tau$ and $\theta$. 

\begin{eg}
Fix some $\theta_0 \in [0, 2\pi)$.  Then the sets $\{\Gamma_{\tau,  \theta_0}\}_{\tau \in [-1, 1]}$ foliate the tiled plane and form a parallel foliation corresponding to any trajectory $\gamma$ with $\theta(\gamma, P_0)=\theta_0$.
\end{eg}

Let us now fix $\tau_0$.  Then the sets $\Gamma_{\tau_0,  \theta}$ overlap,  even inside one tile.  Similarly to the case of the geodesic flow,  we lift them up in another dimension,  in order to avoid intersection. 

\smallskip

Consider the euclidian space $\R^3=\{\left(\textbf{X},  \Theta\right)\}$ as a product of a tiled plane $\R^2$ with a coordinate $\textbf{X}$ on it,  and of an orthogonal line with a coordinate $\Theta$.  We now define a set $\widehat{S_{\mathcal{P}}^{\tau}}$ as a union of its horizontal sections: $\widehat{S_{\mathcal{P}}^{\tau}}:= \cup_{\theta \in [0,2\pi)} \Gamma_{\tau, \theta}$.  It is obviously a piecewise smooth surface.  

The surface $\widehat{S_{\mathcal{P}}^{\tau}}$ is built from many gradually turning «stairs» of trajectories with the same energy parameter.  Any tiling billiard trajectory $\gamma$ on the tiling $\mathcal{P}$ is a connected component of a horizontal section of the helicoid,  its height is defined by the angle parameter.  We call such a  surface $\widehat{S_{\mathcal{P}}^{\tau}}$ (depending strongly on the underlying tiling) \textbf{Dynnikov's helicoid of energy} $\tau$.  It has $3$ periods that do not depend on $\tau$,  as shows the following

\begin{lem}\label{lem:3-periodic}
Let $\mathcal{P}$ be a triangle or cyclic quadrilateral tiling,  and $\widehat{S_{\mathcal{P}}^{\tau}}$ a corresponding Dynnikov's helicoid of some energy $\tau \in [-\pi,\pi]$.  Then,  $\widehat{S_{\mathcal{P}}^{\tau}}$ is $3$-periodic with periods $V_1:=V_1(\mathcal{P}),  V_2:=V_2(\mathcal{P})$ and $V_3=(0,0,  2\pi)$ given by
\begin{itemize}
\item[1.] $V_1(\mathcal{P})=\left(-\textbf{c},  2\gamma\right),  V_2(\mathcal{P})=\left(-\textbf{a},  2\alpha\right)$ if $\mathcal{P}$ is a triangle tiling,  
\item[2.] $V_1(\mathcal{P})=\left(\textbf{a}+\textbf{b},  2\delta\right),  V_2(\mathcal{P})=\left(\textbf{b}+\textbf{c},  2\alpha\right)$ if  $\mathcal{P}$  is a cyclic quadrilateral tiling.
\end{itemize}
The notations here are consistent with those from paragraph \ref{subsubs:notations}.
\end{lem}

\begin{proof}
Periodicity in vertical direction is obvious since $\Gamma_{\tau, \theta}=\Gamma_{\tau, \theta+2\pi}$.  The rest follows from Lemma \ref{lem:lem_function}.  Indeed,  for two tiles $P_1,  P_2$ of the same color,  the difference of angle parameters is constant and equal to $\varphi(\textbf{v})$ with $\textbf{v}=P_1-P_2$.  One concludes that a point $(\textbf{X},  \Theta) \in \widehat{S_{\mathcal{P}}^{\tau}}$ belongs to a surface if and only if a shifted point $(\textbf{X}',  \Theta'):=(\textbf{X}+\textbf{v},  \Theta-\varphi(\textbf{v})) \in \widehat{S_{\mathcal{P}}^{\tau}}$ does.  By continuity,  the argument follows for differently colored tiles.

The lattice of periods is generated by the vectors $\textbf{v}_1:=\overrightarrow{BA}=-\textbf{c}$ and $\textbf{v}_2:=\overrightarrow{CB}=-\textbf{a}$ for triangle tilings and by the vectors $\textbf{v}_1:=\overrightarrow{AC}=\textbf{a}+\textbf{b}$ and $\textbf{v}_2:=\overrightarrow{BD}=\textbf{b}+\textbf{c}$ for cyclic quadrilateral tilings.  The values of $\varphi$ are given on Figure \ref{fig:quasi_periodic}.  
\end{proof}

\begin{figure}
\centering
\includegraphics[scale=0.2]{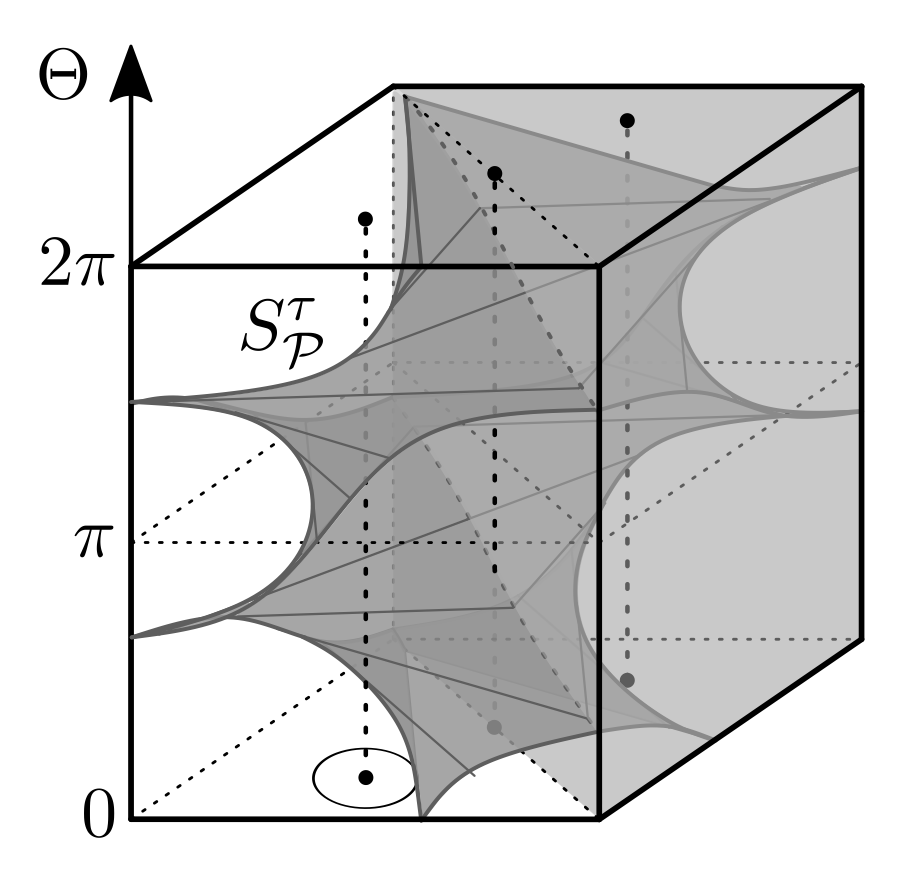}
\caption{The compact ruled surface $S_{\mathcal{P}}^{\tau}$ obtained as an intersection of Dynnikov's helicoid with a torus obtained by gluing parallel faces of the prism $\mathcal{D}\times[0, 2\pi)$. }\label{fig:helicoid}
\end{figure}

\subsection{Symmetries of  $S_{\mathcal{P}}^{\tau}$ and link to interval exchange maps}\label{subs:sym_and_IET}
We now describe some properties of Dynnikov's helicoid.  Consider a fundamental domain $\mathcal{D}_m$ of a $P$-tiling, as in Section \ref{sec:folding} and on Figure \ref{fig:symmetry}.  Fix a horizontal direction as the direction needed for the definition of the angle parameter as that of the edge containing $m$.  

Lemma \ref{lem:3-periodic} implies that the intersection with the prism $\mathcal{D}_m \times [0, 2\pi) \cap \widehat{S_{\mathcal{P}}^{\tau}}$ is a fundamental domain of the surface $\widehat{S_{\mathcal{P}}^{\tau}}$.  Under identification of the borders under shifts $V_j,  j=1,2,3$,  such a prism becomes a $3$-torus $\mathbb{T}^3_{\mathcal{P}}$. 
Let $\pi: \R^3 \rightarrow \mathbb{T}^3_{\mathcal{P}}$ be a corresponding projection.  Denote by $S:=S_{\mathcal{P}}^{\tau}=\pi\left( \widehat{S_{\mathcal{P}}^{\tau}}\right)$ a compact surface, represented on Figure \ref{fig:helicoid}.

\begin{lem}\label{lem:topology_calculations}
Let $\tau \in [-\pi,\pi]$,  and $\mathcal{P}$ be a triangle (cyclic quadrilateral) tiling.   For a corresponding Dynnikov's helicoid $\widehat{S}:=\widehat{S_{\mathcal{P}}^{\tau}}$ and $S:=\pi(\widehat{S})$,  the following holds:

\begin{itemize}
\item[1.] the surface $S$ is centrally symmetric with respect to the point $M:=(m,  \pi)$.  The quotient $S^M$ under the central symmetry with respect to $M$,  is a non-orientable surface; moreover,  this point $M$ belongs to $S$ if and only if $\tau=0$;
\item[2.] if $\tau=0$,  the surface $S$ has an additional symmetry under the map $\textbf{s}: \Theta \mapsto \Theta+\pi$.  The quotient
$^{S}/_{\textbf{s}}$ under this symmetry is homeomorphic to the projective plane $\mathbb{P}^2(\R)$.  Moreover,  for tiles containing its circumcenter,  the foliation on the projective plane induced by a horizontal foliation on $S$ is a foliation with one $3$ (or $4$)-prong singularity and three (or four) $1$-prong singularities;


\item[3.] if $P$ contains its circumcenter\footnote{for triangles it is equivalent to the acuteness} and $\tau=0$,  the genus of $S$ is equal to $3$.  For $P$ a triangle,  $S$ has two double saddles,  exchanged by $\textbf{s}$.  If $P$ is a cyclic quadrilateral,  the surface $S$ has four simple saddles,  exchanged in pairs via $\textbf{s}$.  If $P$ doesn't contain its circumcenter,  the genus of $S$ is equal to $1$.
\end{itemize}
\end{lem}

\begin{figure}
\centering
\includegraphics[scale=0.5]{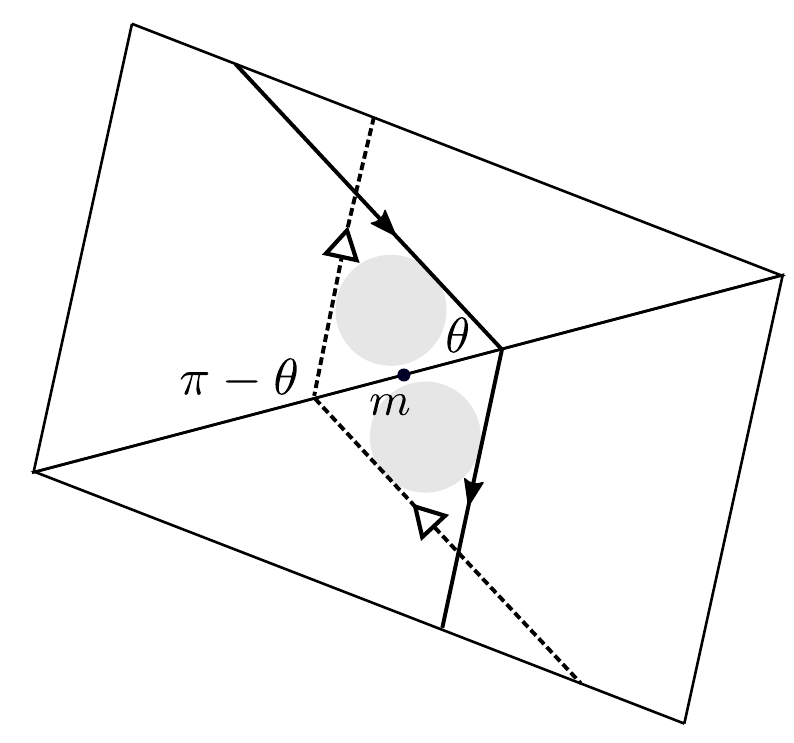}
\caption{Fundamental domain $\mathcal{D}_m$ of a triangle tiling and two trajectories passing through it,  $\gamma$ and $\gamma'$.  Each of the trajectories is tangent to both circles centered at circumcenter of radius $\tau(\gamma)=\tau(\gamma')=\tau$.  Moreover,  $\theta(\gamma)+\theta(\gamma')=\pi$.}\label{fig:symmetry}
\end{figure}

\begin{proof}
All of this is direct.  The first statement follows from the symmetry of $\mathcal{D}_m$.  If there exists a trajectory $\gamma$ crossing $\mathcal{D}_m \times \{\theta\}$,  there exists a trajectory $\gamma'$ crossing $\mathcal{D}_m \times \{\pi-\theta\}$,  see Figure \ref{fig:symmetry}.  If $\tau=0$,  the trajectory passing by $m$ and orthogonal to the edge that contains it,  belongs to $S$,  and $\gamma$ coincides with $\gamma'$ above (and has opposite orientation).  The  second statement follows from the existence of the additional symmetry and the calculation of Euler characteristic. 
In this particular case,  the interior and exterior of $S^0_{\mathcal{P}}$ in $\mathbb{T}^3_{\mathcal{P}}$ are isometric.

Suppose now that $P$ contains its circumcenter and $\tau=0$.  We calculate $\chi(S)$ as a sum of indices of its singular points with respect to the height function $\Theta$.  If $P$ is a triangle,  all of the vertices of $\mathcal{D}_m$ are identified under lattice action.  The surface $S$ has then two monkey saddles (index $-2$) corresponding to the angle parameters $\theta$ and $\theta+\pi$,  where the trajectories enter (or get out from) the vertex.  This gives $\chi=-4$ and $g=3$.  Analogous calculation may be done for «acute» cyclic quadrilaterals : $\mathcal{D}_m$ has two vertices,  modulo the action of the lattice.  They give then $4$ simple saddles on layers $\theta_1,  \theta_2,  \theta_1+\pi,  \theta_2+\pi$,  each of index $-1$.  Then,  once again,  $g=3$.  For the obtuse case,  the same calculation gives $g=1$. This proves the third point.
\end{proof}

\begin{rem}
It is not surprising that genus may fall drastically when the parameters change continuously since the helicoid is parametrically defined by $\theta$ and $\tau$.  A good exercise is to understand the change of genus of $S$ when $\tau$ changes.
\end{rem}

The surfaces $\widehat{S}$ and $S$ are equipped with natural oriented foliations by tiling billiard trajectories.  The intersection of $S$ with the border of prism $\mathcal{D}_m \times[0, 2\pi)$ consists of a circle,  and the first return map of the tiling billiard flow on this circle is an interval exchange transformation $T=T(P,  \tau)$ with $6$ or,  in case of quadrilaterals,  $8$ intervals of continuity.  If one passes to the quotient $S^M$,  such first return map is reduced to an interval exchange transformation $F(P,  \tau)$ of $3$ or $4$ intervals of continuity with flips,  and $T=F^2$.  The study of tiling billiard dynamics in triangle and cyclic quadrilateral tilings is then reduced to the study of parametric families of interval exchange transformations with flips.  This is the important leitmotiv of the works \cite{Davis2018, HubertPaRo2019,  PaRo2019} for triangles and of the work \cite{DHMPRS2020} for quadrilaterals.  

\section{Novikov's problem and tiling billiards}\label{sec:Novikov}

The helicoid construction from Section \ref{sec:Dynnikov} is elementary but crucial since it connects the dynamics of tiling billiards with a classical topology problem formulated in $1982$ by Novikov.  It concerns the level sets of quasiperiodic functions on the plane with $3$ quasi-periods and has important motivations coming from physics of metal conductivity.  Nowadays,  the interest to this problem is vivid in both mathematics and physics.  Novikov's problem is a field in itself and we do not aim to give an overview nor a bibliography of this rich subject.  Our goal is to point out a new connection -- the one with tiling billiard systems that,  hopefully,  can shed some light on the problem in itself.

\subsection{Novikov's problem : statement and generic behaviour}\label{subs:novikovs}

Consider a piecewise smooth function $f: \mathbb{T}^3 \rightarrow \R, 
\mathbb{T}^3=\R^3/\Z^3$.  Without loss of generality we suppose that $f$ takes values in the interval $[-\pi,  \pi]$.  Let $M_{\tau}=f^{-1}(\tau)$ be its level surface and $\widehat{M}=\pi^{-1}(M)$ be the $\Z^3$-covering in $\R^3$.  Here $\pi: \R^3 \rightarrow \mathbb{T}^3$ is a standard projection. 

\smallskip

\textbf{Novikov's problem.} Fix a covector $H=(H_1,  H_2, H_3) \in \mathbb{P}^2(\R)$.  Study the behavior of connected components of plane sections of $\widehat{M}$ by a family of parallel planes $H_1 x_1+ H_2 x_2+H_3x_3=\mathrm{const}$.

The corresponding parallel plane sections define an orientable foliation $\mathcal{F}$ on $M$.  We are interested in the closures of its leaves.  

For simplicity we suppose that the covector $H$ is totally irrational.  As Dynnikov showed in \cite{Dynnikov1999TheGO},  three qualitative behaviors are possible: trivial,  integrable and chaotic.  \textbf{Trivial behavior} means that all components of all $H$-sections are compact.  \textbf{Integrable behavior} means that all regular non-closed components are confined in bands of the plane,  i.e.  have an asymtptotic direction.  Integrable behavior corresponds to the decomposition of $M$ into cylinders of closed trajectories and tori (possibly,  with holes) on which $\mathcal{F}$ winds in a way that it is topologically equivalent to  irrational rotation.  Finally,  \textbf{chaotic behavior} means that the closure of some leaf of $\mathcal{F}$ coincides with a component of $M$ of genus at least $3$.  One of the main results in \cite{Dynnikov1999TheGO} is that the chaotic behavior occurs in a very rare number of cases,  see Theorems $1$ and $3$ there. 

\begin{thm}[Dynnikov,  \cite{Dynnikov1999TheGO}]\label{thm:codim1}
Fix a piecewise smooth and generic function $f$ and a vector $H$.  Then there exist two values $\tau_1(H), \tau_2(H)$,  $\tau_1 \leq \tau_2$ such that for all $\tau \notin [\tau_1(H),  \tau_2(H)]$ the behavior of corresponding sections of $M_{\tau}$ is trivial,  and for $\tau \in [\tau_1(H),  \tau_2(H)]$ it is integrable.  In the case when $\tau_1(H)=\tau_2(H)$,  the behavior may be chaotic.  Moreover,  the chaotic behavior is rare in the following sense : the set $\mathcal{O}$ of vectors $H$ corresponding to non-chaotic behavior is open and dense in $\mathbb{P}^2(\R)$.  
\end{thm}

\begin{que}\label{que:que1}
The question of whether the set $\mathcal{O}$ is of full Lebesgue measure is open,  even for $M$ of genus $3$.  
\end{que}

\begin{rem}\label{remark:number_of_parameters}
The space of couples function-covector $\{(f,  H)\}$ has infinite dimension although the qualitative behavior of sections depends only on the finite number of parameters.  Indeed,  one considers an exact $1$-form induced on $M$ by a linear form $\alpha=dH$ on $\R^3$.  Via Hodge theorem,  take a unique harmonic form $\omega$ on $M$ such that $\omega \in [\alpha]$.  Then $\omega$ defines  a cohomologous,  and even cobordant foliation to $\mathcal{F}$ which has the same global invariants as $\mathcal{F}$ itself.  In other words,  one can straighten out the foliation $\mathcal{F}$ and preserve the class of qualitative behavior (trivial,  integrable or chaotic).  The corresponding surface has a flat metric.  This shows the relationship of Novikov's problem with the dynamics of families of interval exchange transformations. 

Moreover,  the word generic in Theorem \ref{thm:codim1} has to be precised.  We refer our reader to the original article \cite{Dynnikov1999TheGO}.  Some details are discussed at the end of this Section.
\end{rem}

\begin{que}\label{que:que2}
A much less stronger question than Question \ref{que:que1} is open - prove that in the set of pairs (surface,  vector) the set of chaotic couples is of measure $0$.
\end{que}

The most strongest form of such type of questions is a following 

\begin{conj}[Novikov-Maltsev,  $2003$]
For a fixed surface $M$,  the Hausdorff dimension of the set of covectors $H$ admitting chaotic sections is smaller than $2$.
\end{conj}

\subsection{Results on tiling billiards and their topological interpretation}\label{subs:interpretation}

Our goal here is to include the study of tiling billiards in triangle and cyclic quadrilateral tilings into the setting of Novikov's problem.  


Lemma \ref{lem:3-periodic} states that a surface $\widehat{S_{\mathcal{P}}^{\tau}}$ is $3$-periodic,  with the vectors $V_1,  V_2, V_3$ defining the base of the corresponding lattice of symmetries.  Therefore,  there exists a unique linear map $A_{\mathcal{P}} \in \mathrm{SL}_3(\R)$ such that $A_{\mathcal{P}} (V_j)=E_j$,  with $E_j$ forming the standard orthonormal basis in $\R^3,  j=1,2,3$.  Note that $V_3=2 \pi \cdot E_3$.  Denote $ \widehat{M_{\mathcal{P}}^{\tau}}: = A_{\mathcal{P}}\widehat{S_{\mathcal{P}}^{\tau}}$ the rectified surface.  Then,  $M_{\mathcal{P}}^{\tau}:=\pi  (\widehat{M_{\mathcal{P}}^{\tau}})$ is a subsurface of a standard torus.  Traectories of a billiard are the connected components of horizontal sections of Dynnikov's helicoid. Under the linear map,  they map to the connected components of intersections $ \widehat{M_{\mathcal{P}}^{\tau}} \cap A_{\mathcal{P}} \{ \Theta=\theta\}$.  Once a helicoid is constructed,  two points of view differ only by a linear map!

\begin{rem}
The map $A_{\mathcal{P}}$ as well as the direction of the co-vector $H$ defining the planes $A_{\mathcal{P}} \{ \Theta=\theta\}$ doesn't depend on $\tau$ since the vectors $V_j$ do not depend on it.
\end{rem}

\begin{que}\label{que:details}
How large is a class of surfaces described (in terms of Remark \ref{remark:number_of_parameters}) by Dynnikov's helicoids for triangle and cyclic quadrilateral tilings?
\end{que}

The first non-trivial case of Novikov's problem (when chaotic behavior is possible) occurs in genus $g(M)=3$.  By Lemma \ref{lem:topology_calculations},  the maximal genus of Dynnikov's helicoids is equal to $3$ and such helicoids are always centrally symmetric.  We think that this is the only obstruction and that the
answer to the Question \ref{que:details} is : \emph{any centrally symmetric surface of genus }$3$.  A careful dimension count should be done,  see Remark \ref{remark:number_of_parameters}.

\smallskip

Let us now remind some results on the dynamics of considered tiling billiards.  The following has been conjectured in \cite{Davis2018},  first proven in \cite{HubertPaRo2019} and,  finally,  a simpler proof was found in \cite{PaRo2019} via renormalization.  As we have recently discovered,  the analogous proof,  even if in a different setting,  has already been provided in $1989$ by Meester and Nowicki in \cite{MeNo}.\footnote{Meester and Nowicki consider a percolation model on the circle which is exactly the circumcircle appearing via folding.  Their model is defined by drawing a chord in a circle and coloring all vertices in the set $\{\alpha n + \beta m, m,n \in \Z\} \subset \mathbb{S}^1$ on the left of the chord in one color,  and others in another color.  The corresponding coloring of the lattice $\Z^2$ with coordinates $(m,n)$ has open one-colored clusters. They correspond to escaping trajectories of triangle tiling billiards that fold into the initial chord.}

\begin{thm}[\cite{HubertPaRo2019, PaRo2019,  AvilaSkripchenkoHubert}]
\label{thm:n=3}
For a tiling billiard in a $P$-tiling defined by a triangle $P$,  the following holds :
\begin{enumerate}
\item[1.] For almost any $P$,  all trajectories are either periodic or linearly escaping.  
\item[2.] If a trajectory $\gamma$ escapes in a non-linear way then,  necessarily,  $\tau(\gamma)=0$ (it passes by circumcenters of all crossed tiles) and $P \in \mathcal{R}$.  Here $\mathcal{R}$ is a fractal set defined by an explicit continued fraction algorithm.  This set has zero measure and $\dim_H \mathcal{R} \in (1,2)$. 
\end{enumerate}
\end{thm}

This set $\mathcal{R}$ is the \textbf{Rauzy gasket}.  We refer to \cite{ArnouxStarosta2013} for the classic definition of the Rauzy gasket.  Many other definitions have been given throughout the last fourty years,  related to the circle percolation (Meester-Nowicki \cite{MeNo}), dynamics of Arnoux-Rauzy family of interval exchange transformations (Arnoux-Rauzy \cite{AR}),  systems of isometries (Dynnikov-Skripchenko \cite{DS_band}),  Novikov's sections of some polyhedral object (Dynnikov-DeLeo \cite{DeLeo_2008}),  and,  as shown here,  in relation to the dynamics of non-linearly escaping trajectories of tiling billiards (Davis et al.  \cite{Davis2018},  Hubert and ourselves \cite{HubertPaRo2019,  PaRo2019}).  In the last years the understanding emmerged that all these different interpretations of the Rauzy gasket are equivalent,  as is implied in this work and shown in the works \cite{DHS}  and \cite {HubertPaRo2019}.  

\begin{rem}
The calculation of the Hausdorff dimension of the Rauzy gasket is highly non-trivial.  The first estimate $\dim_H \mathcal{R}<2$ was obtained in \cite{AvilaSkripchenkoHubert}.  Nowadays more refined norms exist and $\dim_H$ is confined into the interval $(1.19,  1.825)$.  The lower bound was obtained by Gutiérrez-Romo and Matheus in \cite{GutierrezMatheus19},  the upper bound is obtained by combining the arguments of Gamburd-Magee-Ronan \cite{Gamburd} with estimates by Baragar \cite{Baragar1998},  as was recently explained in \cite{Fougeron20}.
\end{rem}

\smallskip

Recently,  our French-Russian team managed to prove the following result,  analogous to Theorem \ref{thm:n=3} result for cyclic quadrilateral tilings.

\begin{thm}[\cite{DHMPRS2020}]\label{thm:n=4}
For a tiling billiard in a $P$-tiling defined by a cyclic quadrilateral $P$,  the following holds :
\begin{enumerate}
\item[1.] 
For almost any $P$,  all trajectories are either periodic or linearly escaping.  
\item[2.]
If a trajectory $\gamma$ escapes in a non-linear way then,  necessarily,  $\tau(\gamma)=0$ and $P \in \mathcal{N}$.  Here the set $\mathcal{N}$ is a fractal set defined by an explicit continued fraction algorithm,  has zero measure and $\dim_H \mathcal{N} <3$. 
\end{enumerate}
\end{thm}

 The first difficulty in proving this Theorem was to find a renormalization process in order to define the algorithm that constructs $\mathcal{N}$ via a continued fraction algorithm.  Once this was done,  the main technical difficulty consisted in proving the ergodic properties of such an algorithm.  The first part is combinatorial and generalises the methods in \cite{PaRo2019},  the second part is  based on the thermodynamic formalism,  elaborated recently by Fougeron \cite{Fougeron20}.  
 
Modulo the regularity details described in Question \ref{que:details} that have to be figured out,  Theorem \ref{thm:n=4} solves an open case of Novikov's problem since it characterizes chaotic directions for symmetric genus $3$ surfaces.  In relation to this interpretation,  the set $\mathcal{N}$ of cyclic quadrilaterals that may exhibit non-linear escaping trajectories is called the \textbf{Novikov's gasket.}

\bigskip

The renormalization methods proposed in \cite{PaRo2019,  DHMPRS2020} prove the first points of both Theorems \ref{thm:n=3} and \ref{thm:n=4}.  Although one can see that these first points follow, \emph{naively},  from Dynnikov's result of 1999,  namely Theorem \ref{thm:codim1} here.  Indeed,  since for a fixed direction $H$,  the chaotic behavior can only happen for one energy parameter,  the central symmetry of $M_{\tau}$ implies that this parameter is exactly $\tau=0$.  

Although,  the consequence is,  as we said,  not completely precise.  The regularity is problematic.  Indeed,  consider a function $f$ corresponding to some Dynnikov's helicoid for triangle tilings.  Then $f$ is not generic in Dynnikov's sense : its level surfaces have double saddle points and the work \cite{Dynnikov1999TheGO} works with surfaces admitting only simple saddles.  Although,  the explicit calculations of chaotic sections were done for this case in \cite{PaRo2019} and in \cite{DeLeo_2008},  and they finalise the proof of Theorem \ref{thm:n=3}.

For cyclic quadrilaterals,  Dynnikov's results do apply since the surfaces are generic enough and the corresponding saddles are simple.  The point 1.  of Theorem \ref{thm:n=4} is then a direct consequence of Theorem \ref{thm:codim1}.  We wonder if the genericity assumptions of Theorem \ref{thm:codim1} in \cite{Dynnikov1999TheGO} could be weakened in order to apply directly to surfaces with saddles of higher multiplicity. 

\begin{rem}
Triangle tilings can be seen as degenerations of cyclic quadrilateral tilings with two vertices of a quadrilateral approaching by following the arc of the circumcircle of a tile.  The corresponding surfaces are degenerations of a more general case : saddle points collide in a double (monkey) saddle.
\end{rem}

\subsection{Tree Conjecture for cyclic quadrilateral tilings}\label{subs:Tree Conjecture 4}
In this paragraph,  we advance towards the understanding of the symbolic dynamics of cyclic quadrilateral tilings,  analogous to Theorem \ref{thm:tree_conjecture}.

\begin{conj}[Tree Conjecture for cyclic quadrilateral tiling billiards]\label{conj:TCforquadrilaterals}
Any periodic trajectory $\gamma$ in cyclic quadrilateral tiling doesn't contour tiles,  i.e.  the domain $\Omega_{\gamma}$ bounded by it doesn't contain a full tile.
\end{conj}

The symbolic behavior of any periodic trajectory $\gamma$ is defined by the behavior of singular trajectories in its parallel foliation inside $\Omega_{\gamma}$.  Using this idea,  we have shown in \cite{PaRo2019} that for any locally foldable tiling,  the Tree Conjecture is equivalent to the following Bounded Flower Conjecture dealing with only singular trajectories (or \textbf{petals}).

\textbf{Bounded Flower Conjecture.} Any singular periodic trajectory $\gamma$ passing by a vertex $v$ of a tiling satisfies the two following properties : first,  it intersects two neigboring tiles $P$ and $P_e$,  $P \cap P_e=e$; second,  $e \in \Omega_{\gamma}$.

\smallskip

Let us include a petal $\gamma$ in its parallel foliation.  If there is another petal $\gamma'$ passing by the same vertex $v$,  then we can prove that the Bounded Flower Conjecture for $\gamma$ holds.  If it doesn't,  it would mean that either $\Omega_{\gamma} \subset \Omega_{\gamma'}$ or $\Omega_{\gamma'} \subset \Omega_{\gamma}$,  and two trajectories have opposite orientations. Then there exist two periodic trajectories of the same energy $\tau$ in the parallel foliation turning in different senses.  It means that a corresponding helicoid $\widehat{S_{\mathcal{P}}^{\tau}}$ has a section with two connected components,  one of electron type,  and one of hole type in the terminology of \cite{Dynnikov1999TheGO}.  This and the connectedness of $\widehat{S_{\mathcal{P}}^{\tau}}$ would imply that the surface $S_{\mathcal{P}}^{\tau}$ has genus at least $4$ which brings a contradiction with Lemma
\ref{lem:topology_calculations}. 

Unfortunately,  we were not yet able to eliminate the sitation when $\gamma$ is an only singular trajectory passing by $v$ in its parallel foliation,  giving the obstructions to Bounded Flower Conjecture.  It would mean that either $\gamma$ doesn't pass by $P$ and $P_e$,  or it does but with $e \notin \Omega_{\gamma}$.  In the case of triangle tilings,  the two cases were eliminated by using the additional symmetries which are not anymore present for quadrilaterals.  The symbolic dynamics for quadrilateral tiling billiards is more complicated and still needs to be understood in more detail. 

\section{Perspectives for higher genus}\label{sec:generalisations_last}
The ideas of Section \ref{sec:folding} apply to any locally foldable tiling: parallel foliations exist,  and all bounded trajectories are periodic and stable.  Moreover,  the quasiperiodicity observed in Lemma \ref{lem:lem_function} is present in any locally foldable periodic tiling.

In this Section we show why the helicoid of Section \ref{sec:Dynnikov}, constructed there for triangle and cyclic quadrilateral tilings,  can be constructed for many other locally foldable polygonal \emph{periodic} tilings. 

\begin{rem}[Combinatorial data of a locally foldable periodic tiling]
Any locally foldable polygonal periodic tiling defines a bipartite graph $G$ on the $2$-torus $\mathbb{T}^2$.  Indeed,  we consider a dual graph (a graph of faces of the tiling) : two tiles are connected if and only if they have a common edge  in the tiling,  see Figure \ref{fig:graphs} for two examples.  By periodicity,  this graph factors to the graph $G$ on the torus.  Since the locally foldable tiling is $2$-colorable,  $G$ is bipartite. Of course,  one such graph $G$ defines a family of corresponding locally foldable tilings.  We say that these tilings have \textbf{the same combinatorics.} 
\end{rem}

The apparent difficulty in the realization of a helicoidal construction for a general locally foldable periodic tiling is that it is not clear what should be the energy parameter $\tau$.  In the point 1.  of Theorem \ref{thm:first_properties} $\tau$ was defined as a distance to a circumcenter of a triangle or cyclic quadrilateral.  Here we point out an approach that gives a substitute to this circumcenter in the general case.  For this,  we use recent results obtained on locally foldable tilings in relationship to the study of dimers.

It is non-trivial to describe the set of parameters of locally foldable periodic polygonal tilings of fixed combinatorics $G$.  This can been done by following recent works on dimers,  in particular that by Kenyon,  Lam,  Ramassamy and Russkikh in \cite{Kenyon2018DimersAC}.  There the authors show that there is a bijection beween such tilings and \emph{liquid phase dimer models}.  This bijection uses the beautiful connection of the dimer model with the complex algebraic curves and their amoebas established by Kenyon,  Okounkov and Sheffield.  It happens that the studied curves are of a very special type,  namely Harnack curves,  we refer to \cite{KOS2006} for more details.

\begin{figure}
\centering
\includegraphics[scale=0.8]{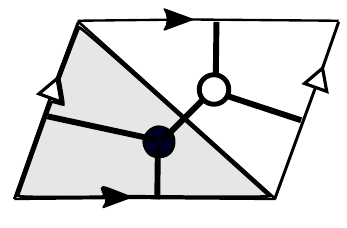}
$\;\;\;\;\;\;\;$
\includegraphics[scale=0.7]{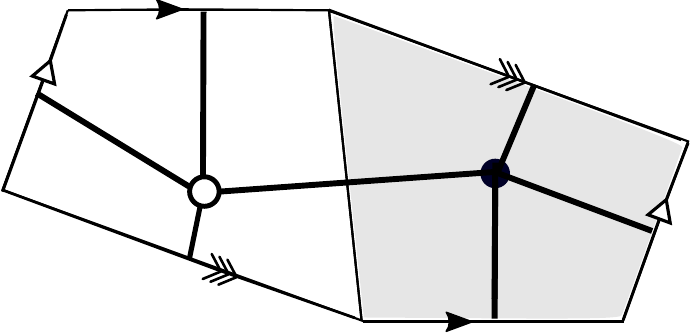}
\caption{Bipartitie graphs on the torus $\mathbb{T}^2$
defining the combinatorics of  triangle and quadrilateral tilings.  Both graphs have $2$ vertices,  the triangle graph has $3$ edges,  the cyclic quadrilateral graph has $4$ edges.  The torus is given as regluing the opposite parallel sides of the fundamental domain of the tiling.  }\label{fig:graphs}
\end{figure}

From all this important theory we use only the fact that \emph{typically} the locally foldable tilings fold into bounded domains.\footnote{This corresponds to the spectral curve only having simple zeroes.} In this case,  take any tile $P$ and all of its copies $gP$ in the tiling,  here $g$ is an element of the lattice of isometries of the tiling.  Then,  the boundedness of the folding implies that for all $g, $ the images of $P$ and $g(P)$ under folding differ by a rotation with some center $C$.\footnote{Indeed,  the orbit of a discrete subgroup $H$ of the affine group is bounded only if $H$ is a subgroup of $\mathrm{SO}(2)$.} Moreover,  the center $C$ can't depend on $P$ in order for the folding to be bounded.  This center $C$ is a point with respect to which we define the energy $\tau$! Once this step is done,  the helicoid construction of Section \ref{sec:Dynnikov} is repeated word by word.  

\begin{que}
Suppose that the helicoid $S_{\mathcal{P}}^{\tau}$ exists for a locally foldable tiling $\mathcal{P}$.  What is its genus as a function of 
$\mathcal{P}$ and $\tau$ ?
\end{que}

\begin{que}
What families $\mathcal{T}_{\mathcal{P}}$ of interval exchange transformations  arise as first-return maps on some well-chosen transversals ? 
\end{que}

We hope to answer these questions in future work,  in order to prove the following 

\begin{conj}\label{conj:probability}
Fix the combinatorics of a periodic $2$-colored tiling of a plane,  defined via a bipartite graph $G$ on the $2$-torus.  Suppose that a tiling with such combinatorics folds into a bounded domain.  Then,  the non-linear escape of trajectories on such tiling is only possible if the trajectories pass by the point $C$.  Otherwise,  the trajectories either escape linearly or are periodic.
\end{conj}

Conjecture \ref{conj:probability} follows \emph{naively} from Theorem \ref{thm:codim1} and helicoidal construction.  We do not announce it as a result since the regularity details have to be thouroughly checked,  as discussed in paragraph \ref{subs:interpretation}.  Indeed,  locally foldable tilings permit mutliplicity in saddles,  although the arguments in \cite{Dynnikov1999TheGO} suppose Morse property.  Nevertheless, we believe that these complications are avoidable.  Moreover,  as follows from dimer model theory,  this Conjecture would apply to an open set of parameters of locally foldable tilings,  maybe even of full measure.

A following much stronger conjecture is a reformulation of Question \ref{que:que2}.

\begin{conj}
The set of parameters of locally foldable tilings of fixed combinatorics $G$ admitting non-linearly escaping trajectories,  has measure zero.
\end{conj}

\begin{que}
Is such set an invariant gasket of some continued fraction algorithm?
\end{que}

\smallskip

We find very exciting a possibility to construct multi-dimensional fractal objects corresponding to every bipartie graph on the torus.  It could make quite a collection! For the graphs,  corresponding to triangle and quadrilateral periodic tilings,  these objects are,  respectively,  the Rauzy and the Novikov gaskets.

\begin{center}
\textbf{ACKNOWLEDGMENTS}
\end{center}
I am grateful to Ivan Dynnikov for his beautiful idea of a helicoid that he shared in a short on-line call during the pandemic,  Section \ref{sec:Dynnikov} is entirely based on it.  I am grateful to Pascal Hubert  and Bruno Sevennec for fruitful discussions on the subject and comments on the preliminary versions of this text.  I am thankful to Dima Chelkak  for introducing me to the dimer model and Benoît Laslier for answering my questions on it and advice.  I am obliged to Théo Marty for the Figure \ref{fig:helicoid} he drew in Inkscape in one evening and to Paul Mercat for his 3D prints of Dynnikov's helicoids.  I am also thankful to my new home,  \emph{Institut de Mathématiques de Marseille,}  and the members of our laboratory for warm and productive atmosphere.  

\smallskip

\bibliographystyle{plain}

\bibliography{references}

\Addresses

\end{document}